\documentclass[12pt]{amsart}

\usepackage[utf8]{inputenc}
\usepackage{amssymb, url}
\usepackage{amsmath, verbatim}
\usepackage{enumerate}
\usepackage{amsfonts}
\usepackage{hyperref}

\newtheorem{theorem}{Theorem}[section]
\newtheorem{lemma}[theorem]{Lemma}
\newtheorem{proposition}[theorem]{Proposition}
\newtheorem{corollary}[theorem]{Corollary}

 \theoremstyle{definition}
 
 \newtheorem{defn}[theorem]{Definition}
 \newtheorem{remark}[theorem]{Remark}

\newcommand{\F}{\mathbb{F}}

\newcommand{\N}{\mathbb{N}}

\newcommand{\Ker}{\operatorname{Ker}}

\DeclareMathOperator{\Gal}{Gal}
\DeclareMathOperator{\gen}{gen}

\newcommand{\GL}{{\mathrm{GL}}}
\newcommand{\PGL}{{\mathrm{PGL}}}

\newcommand{\nMat}[2]{\mathrm{M}_{#2}(#1)}

\begin{document}

\author{Uriya First}
\email{uriya.first@gmail.com}
\thanks{Uriya First  was supported, in part, by a postdoctoral research fellowship from the University of British Columbia.}
\address{Department of Mathematics\\University of Haifa\\ Mount Carmel, Haifa, Israel 31905}

\author{Zinovy Reichstein}
\email{reichst@math.ubc.ca}
\thanks{Zinovy Reichstein was partially supported by
National Sciences and Engineering Research Council of
Canada (NSERC) grant No. 250217-2012.}

\author{Santiago Salazar}
\email{santiago.salazar@zoho.com}
\address{Department of Mathematics\\University of British Columbia\\ BC, Canada V6T 1Z2}
\thanks{Santiago Salazar's work on this project was conducted in the framework of NSERC's Undergraduate Summer 
Research Awards (USRA) program at the University of British Columbia.}  

\title[Generators of separable algebras]{On the number of generators of a separable algebra over a finite field}

\keywords{Finite fields, \'etale algebras, M\"obius inversion, separable algebras, number of generators}
\subjclass[2010]{12E20, 13E15, 16H05, 16P10}

\begin{abstract} Let $F$ be a field
and let $E$ be an \'etale algebra over $F$,
that is, a finite product of finite separable field extensions
$E = F_1 \times \dots \times F_r$. 
The classical primitive element theorem asserts that if $r = 1$, then
$E$ is generated by one element as an $F$-algebra. The same is true for
any $r \geqslant 1$, provided that $F$ is infinite. However, if $F$ is a finite field
and $r \geqslant 2$, the primitive element theorem fails in general. In this paper we 
give a formula for the minimal number of generators of $E$ when $F$ is finite.
We also obtain upper and lower bounds on the number of generators of 
a (not necessarily commutative) separable algebra over a finite field.
\end{abstract}

\maketitle
\tableofcontents

\section{Introduction}
The primitive element theorem asserts that a separable field extension $E/F$ of finite degree 
can be generated by one element; see, e.g.,~\cite[Theorem V.4.6]{lang}. It is natural to ask 
if the same is true for every \'etale algebra $E/F$. Recall that an \'etale algebra 
$E$ over a field $F$ is a finite  product $E = F_1 \times \dots \times F_r$, where each $F_i$ is 
a finite separable field extension of $F$. 
If $F$ is an infinite field, then the primitive element theorem
continues to hold: $E$ is generated by one element as an $F$-algebra; see, e.g., \cite[Proposition 4.1]{fr}.
On the other hand, if
 \begin{equation} \label{e.algebra}
 E := \F_{q^{n_1}} \times \dots \times \F_{q^{n_r}}.
 \end{equation}
is an \'etale algebra over a finite field $F = \F_q$ of $q$ elements, the primitive element theorem may fail.
In this paper we will find the minimal number of generators for $E$ as an $\F_q$-algebra.
We will denote this number by $\gen(E)$.

We will call an \'etale $\F_q$-algebra $E$ in~\eqref{e.algebra} pure if $n_1 = \dots = n_r$.
Any \'etale algebra $E$ can be written 
as a  product $E_1 \times \dots \times E_t$, where each $E_i$ is pure, $E_i \cong (\F_{q^{n_i}})^{r_i}$,
and $n_1, \dots, n_t$ are distinct. In Section~\ref{sec:reduction} we will show
that $\gen(E) = \max \{ \gen(E_1), \dots, \gen(E_t) \}$; see Proposition~\ref{prop.pure}.
This reduces the problem of computing $\gen(E)$ to the case where $E$ is pure.
In this case, we will prove the following formula for $\gen(E)$
in Section~\ref{sec:etale}. 

\begin{theorem} \label{thm.main1}
Let $E = \F_{q^n} \times \dots \times \F_{q^n}$ ($r$ times). Then
$\gen(E)$ is the minimal non-negative integer $g$ such that 
$ r \leqslant \dfrac{1}{n}  \displaystyle{\sum\limits_{d \, | \, n}^{\quad} \mu(d)q^{\frac{gn}{d}}}$.
\end{theorem}

Here the sum is taken over all positive divisors $d$ of $n$, and  $\mu: \N \to \{-1, 0, 1\}$ denotes the
M\"obius function. 
In Section~\ref{sec.explicit} we will prove the following consequence of this formula.

\begin{theorem}\label{thm.main2}
Let $E = \F_{q^n} \times \dots \times \F_{q^n}$ ($r$ times).
Then 
\[\lceil\textstyle\dfrac{1}{n }\log_q(nr) \rceil \leqslant \gen (E)\leqslant 
\lceil\textstyle\dfrac{1}{n }\log_q(nr)\rceil +1\ .\]
\end{theorem}

Here, as usual, $\lceil x \rceil$ denotes the smallest integer $n$ such that $x \leqslant n$.
When $n=1$, $\gen(E)=\lceil \log_qr\rceil$; see Corollary~\ref{cor.n=1}.
For $n>1$ both values for $\gen(E)$ allowed by Theorem~\ref{thm.main2} actually occur;
see Theorem~\ref{thm.main4} below.

More generally, we will be interested in the minimal number of generators $\gen(A)$ of a
(not necessarily commutative) separable algebra $A$. All algebras in this paper will be assumed 
to be associative with $1$. Recall that an algebra $A$ over a field $F$ is called separable
if $A$ is finite-dimensional, semisimple  and its center is an \'etale $F$-algebra.  
If $F$ is an infinite field, then $\gen(A) = 1$ if $A$ is commutative and $\gen(A) = 2$
otherwise; see \cite[Remark 4.4]{fr}. Thus, our question is only of interest 
if $F = \F_q$ is a finite field. In this case, by theorems of Wedderburn 
(see, e.g.,~\cite[Theorem~2.1.8]{rowen} and \cite[Theorem~7.1.11]{rowenII}), 
$A$ is isomorphic to a product of matrix algebras over finite fields, i.e.,
\[
A =  \nMat{\F_{q^{n_1}}}{m_1\times m_1}\times\dots
\times \nMat{\F_{q^{n_t}}}{m_t\times m_t} \, . 
\]
Note that if $m_1 = \dots m_t = 1$, then $A$ is an \'etale algebra. 

Once again, Proposition~\ref{PR:pure-red-separable} reduces the problem of computing $\gen(A)$ to
the case where $A$ is pure, i.e.\ $(m_1,n_1)=\dots=(m_t,n_t)$.
Our main result for pure algebras is as follows.

\begin{theorem}\label{thm.main3}
Let $A = \nMat{\F_{q^n}}{m\times m} \times \dots \times \nMat{\F_{q^n}}{m\times m}$ ($r$ times).
Then 
\[\lceil\dfrac{1}{n m^2}\log_q(C\cdot r) \rceil \leqslant \gen (A)\leqslant 
\lceil\dfrac{1}{n m^2}\log_q(C\cdot r)\rceil +1\ ,\]
where  $C:=n\cdot |\PGL_m(\F_{q^n})|=\dfrac{n\prod_{i=0}^{m-1}(q^{nm}-q^{ni})}{q^n-1}$.
\end{theorem}

When $m = 1$, the constant $C$ is $n$ and Theorem~\ref{thm.main3} reduces to
Theorem~\ref{thm.main2}. Note, however, that our proof of Theorem~\ref{thm.main3} relies 
on Theorem~\ref{thm.main2}.

In the case where $A$ is non-commutative, we do not have an explicit formula for the value of $\gen(A)$, 
analogous to Theorem~\ref{thm.main1}; see Remark~\ref{rem.assoc}. However,
our final result, proved in Section~\ref{sec:I},
estimates how frequently each of the two values 
for $\gen(A)$ allowed by Theorem~\ref{thm.main3} is assumed.

\begin{theorem} \label{thm.main4} Fix positive integers $n$ and $m$, and a prime power $q$.
Set
\[ \text{$A_r = \nMat{\F_{q^n}}{m\times m} \times \dots \times \nMat{\F_{q^n}}{m\times m}$ ($r$ times)} \]
and  
let $C$ be as in Theorem~\ref{thm.main3}. 
Let $I_0(g)$ denote the set of integers $r$ such that
$\gen(A_r) = g = \lceil\dfrac{1}{n m^2}\log_q(C\cdot r)\rceil$ and let $I_1(g)$ denote the set of integers $r$ such that
$\gen(A_r) = g = \lceil\dfrac{1}{n m^2}\log_q(C\cdot r)\rceil + 1$. 
Then:
\begin{enumerate}
	\item[{\rm(a)}] $I_0(g)\sqcup I_1(g+1)=\N\cap (C^{-1}q^{(g-1)nm^2},C^{-1}q^{gnm^2}]$
	\item[{\rm(b)}] $|I_0(g)|= C^{-1}(q^{gnm^2}-q^{(g-1)nm^2})(1-O(q^{-g}))$ as a function of $g$.
	\item[{\rm(c)}] 
	$|I_1(g)| \geqslant \lfloor C^{-1} q^{(g-1)m^2} \rfloor$, if $n \geqslant 2$.
	\item[{\rm(d)}] 
	$|I_1(g)|\geqslant \lfloor C^{-1} q^{(g-1)n(m^2 - m + 1)} \rfloor $, if $m \geqslant 2$.
\end{enumerate}
\end{theorem}
Here, as usual, $\lfloor x \rfloor$ denotes the largest integer $n$ such that $n \leqslant x$.

If $(n, m) = (1,1)$, then $I_1(g) = \emptyset$ for every $g$; see Corollary~\ref{cor.n=1}.
If $(n, m) \neq (1, 1)$, then Theorem~\ref{thm.main4} tells us that for any sufficiently 
large integer  $g$, $I_0(g)$ and $I_1(g)$ are both non-empty. In other words, for each sufficiently large $g$,
there exist integers $r_1$ and $r_2$ such that 
\[ 
\gen(A_{r_1}) = g = \lceil\dfrac{1}{n m^2}\log_q(C\cdot r_1)\rceil
\quad\text{and}\quad
\gen(A_{r_2}) = g = \lceil\dfrac{1}{n m^2}\log_q(C\cdot r_2)\rceil + 1.
\]
On the other hand, if we let $r$ range over the interval $[1, R]$, then the probability that
$\gen(A_r) = \lceil\dfrac{1}{n m^2}\log_q(C\cdot r) \rceil$ rapidly approaches $1$ as $R$ increases.

\section{Reduction to the case of pure algebras}
\label{sec:reduction}

We begin with the following well-known version of
the Chinese Remainder Theorem. For lack of a suitable reference 
we include a proof of the implication (a) $\Longrightarrow$ (b).

\begin{proposition}[Chinese Remainder Theorem]\label{lem.crt}
	Let $R$ be a (not necessarily commutative) ring and let $I_1,\dots,I_t \subset R$ be 
	two-sided ideals. Then the following conditions are equivalent:
	\begin{enumerate}
	\item[{\rm(a)}] The natural homomorphism $f \colon R \to R/I_1 \times \dots \times R/I_t$ is surjective.
Here the $j$-th component of $f(r)$ is $r \!\pmod{I_j}$.
	\item[{\rm(b)}]
	$I_1,\dots,I_t$ are pairwise coprime, i.e., $I_i+I_j=R$ for any $i\neq j$.
	\end{enumerate}
\end{proposition}

\begin{proof}
(a) $\Longrightarrow$ (b): By symmetry, it suffices to show that $I_1 + I_2 = R$.
Since $f$ is surjective, there exists an $r \in R$ such that
$f(r) = (1, 0, \dots, 0)$. In particular, $r \in I_2$ by the definition of $f$.
Similarly, $f(1-r) = (0, 1, \dots, 1)$, so $1-r$ lies in $I_1$. Since
$1 = (1- r) + r \in I_1 + I_2$, we conclude that $I_1 + I_2 = R$, as desired.

(b) $\Longrightarrow$ (a): See, e.g., \cite[Proposition~2.2.1]{rowen}.
\end{proof}

In the sequel, $P_g := \F_q[x_1, \dots, x_g]$ and
$R_g:=\F_q\langle X_1, \dots, X_g\rangle$ will denote, respectively, the
commutative polynomial algebra and the 
free associative algebra on $g$ generators over $\F_q$.

\begin{lemma} \label{lem.ideals}
{\rm(a)} A separable algebra $A = \nMat{\F_{q^{n_1}}}{m_1\times m_1} \times \dots \times 
\nMat{\F_{q^{n_r}}}{m_r \times m_r}$ can be generated by $g$ elements over $\F_q$ if and only if
the free associative algebra $R_g$ has $r$ distinct two-sided 
ideals $I_1, \dots, I_r$ such that $R_g/I_i$ is isomorphic to $\nMat{\F_{q^{n_i}}}{m_i\times m_i}$ 
as $\F_q$-algebras for every $i = 1, \dots, r$.

{\rm(b)} An \'etale algebra $E=\F_{q^{n_1}}\times\dots\times \F_{q^{n_r}}$
can be generated by $g$ elements over $\F_q$ if and only if
the polynomial algebra $P_g$ has $r$ distinct ideals $J_1, \dots, J_r$ such that
$P_g/J_i \cong \F_{q^{n_i}}$ as $\F_q$-algebras for every $i = 1, \dots, r$.
\end{lemma}

\begin{proof} 
(a)	Suppose $a_1,\dots,a_g\in A$ generate $A$ over $\F_q$.
	Then the $\F_q$-algebra homomorphism
	$R_g\to A$ sending $X_j$ to $a_j$ ($j = 1, \dots, g$) is surjective.
	Let $I_i$ denote the kernel of the composition $R_g\to A \to \nMat{\F_{q^{n_i}}}{m_i\times m_i}$.
	Then $R_g/I_i\cong \nMat{\F_{q^{n_i}}}{m_i\times m_i}$. Moreover, the ideals $I_1,\dots, I_t$
	are pairwise coprime (and in particular, distinct) by Lemma~\ref{lem.crt}.
	
	Conversely, suppose $I_1,\dots,I_r$ are as above. Then
	$I_1, \dots, I_r$ are maximal and distinct, hence they are pairwise coprime.
	By Lemma~\ref{lem.crt}, the homomorphism $R_g \to \prod_i R_g/I_i$
	is surjective. Since $R_g/I_i\cong \nMat{\F_{q^{n_i}}}{m_i\times m_i}$,	we get an $\F_q$-algebra epimorphism 
	$R_g\to \prod_i  \nMat{\F_{q^{n_i}}}{m_i\times m_i} = A$.
	Hence, $E$ is generated by $g$ elements as an $\F_q$-algebra.
	
Part (b) is proved by the same argument as (a), with the free associative algebra $R_g$ replaced by the commutative
polynomial algebra $P_g$.
\end{proof}

\begin{proposition} \label{prop.pure}
\label{PR:pure-red-separable}
Suppose $A = A_1 \times\dots\times A_t$, where each
factor is a pure separable $\F_q$-algebra, $A_i = \nMat{\F_{q^{n_i}}}{m_i\times m_i}^{r_i}$.
Assume further that the pairs $(m_i,n_i)$ are distinct for $i = 1, \dots, t$.  Then
$\gen(A) = \max \{ \gen(A_1), \dots, \gen(A_t) \}$.
\end{proposition}

 \begin{proof} Let $g = \max \{\gen(A_i) \, |\, i = 1,\dots, t\}$.
 Clearly $\gen(A) \geqslant \gen(A_i)$ for each $i$, and thus 
 $\gen(A) \geqslant g$.
 
 To prove the opposite inequality, note that
 by Lemma~\ref{lem.ideals} there exist $r_i$ distinct two-sided 
 ideals $I_{i, 1}, I_{i, 2}, \ldots, I_{i, r_i}$
 such that 
 \begin{equation} R_g/I_{i, j} \cong \nMat{\F_{q^{n_i}}}{m_i\times m_i}
 \label{e.quotient} \end{equation}
 for each $j = 1, 2, \ldots, r_i$.
 Letting $i$ vary from $1$ to $t$, we obtain $r_1 + \dots + r_t$ ideals,
 $I_{i, j}$. We claim that these ideals are distinct. If we can prove this claim,
 then Lemma~\ref{lem.ideals} will tell us that $A$ is generated by $g$ elements, and the proof
 of Proposition~\ref{prop.pure} will be complete.
 
 To prove the claim, suppose $I_{i, j} = I_{i', j'}$ for some 
 $i,j,i',j'$. Then $i = i'$ by \eqref{e.quotient}, and 
 $j = j'$ because the ideals $I_{i, 1}, I_{i, 2}, \dots, I_{i, r_i}$ 
 were chosen to be distinct. This proves the claim.
 \end{proof}
 
\section{Proof of Theorem~\ref{thm.main1}} 
\label{sec:etale}

\begin{defn} \label{def.Nqn(g)}
In the sequel, $N_{q, n}(g)$ will denote the number of maximal ideals $I$ in the polynomial ring $P_g := \F_q[x_1, \dots, x_g]$
such that $P_g/I \cong \F_{q^n}$.
\end{defn}

We will often fix $q$ and $n$, and treat $N_{q, n}(g)$ as a function $g$.
The symbol $N_{q, n}(g)$ emphasizes this point of view.

Let $E = (\F_{q^n})^r$ be a pure \'etale algebra over $\F_q$. 
By Lemma~\ref{lem.ideals}(b), $\gen(E)$ is the minimal integer $g$ such that 
$r \leqslant N_{q,n}(g)$. Thus in order to prove Theorem~\ref{thm.main1},
it suffices to establish the following formula for $N_{q,n}(g)$. 

\begin{proposition} \label{prop.Mobius}
$N_{q,n}(g) = \dfrac{1}{n} \, {\Large \sum\limits_{d\mid n}} \; \mu(d)q^{\frac{gn}{d}}$.
\end{proposition}

Here  $\mu$ denotes the M\"obius function. Recall that $\mu: \N \to \{-1, 0, 1\}$ is defined
as follows: $\mu(m) = (-1)^j$, if $m$ is the product of $j \geqslant 0$ distinct primes, and
$\mu(m) = 0$, if $m$ is divisible by $p^2$ for some prime $p$.

When $g = 1$, ideals $I$ of $P_1=\F_q[x]$ such that $\F_q[x]/I \cong \F_{q^n}$
are in bijection with monic irreducible polynomials of degree
$n$ in $\F_q[x]$. In this case, Proposition~\ref{prop.Mobius} reduces to the well-known
formula for the number of such polynomilas. The proof of this well-known formula
relies on M\"obius inversion; see, e.g.,~\cite[Section 3.2]{ln} or
\cite[p.~254]{lang}. Our proof of Proposition~\ref{prop.Mobius} proceeds 
along similar lines.

\begin{lemma} \label{lem.ideal-count} 
Let $P_g = \F_q[x_1, \ldots, x_g]$.
The following three sets are in (pairwise) bijective correspondence. In particular, each of these sets has 
cardinality $N_{q,n}(g)$.
\begin{enumerate}
	\item[{\rm(a)}]
	The set of ideals $I \subset P_g$ such that $P_g/I \cong \F_{q^n}$, 
	\item[{\rm(b)}] The set of orbits of
	$\F_q$-algebra epimorphisms $\phi:P_g\to \F_{q^n}$ under the
	action of $\Gal(\F_{q^n}/\F_q)$ given by $\sigma \colon \phi \mapsto \sigma \circ \phi$,
	for any $\sigma \in \Gal(\F_{q^n}/\F_q)$.
	\item[{\rm(c)}] 
    The set of $\Gal(\F_{q^n}/\F_q)$-orbits of order $n$ in $(\F_{q^n})^g$ or equivalently, the set of
    $g$-tuples $(a_1, \dots, a_g) \in (\F_{q^n})^g$ such that $\F_q[a_1, \dots, a_g] = \F_{q^n}$.
\end{enumerate}
\end{lemma}

\begin{proof} 
	The bijective correspondence between the sets (a) and (b) is given by sending
	the $\Gal(\F_{q^n}/\F_q)$-orbit of $\phi:P_g\to \F_{q^n}$
	to $\ker (\phi)$.
	In the other direction, send an ideal $I \subset P_g$ in (a) to the $\Gal(\F_{q^n}/\F_q)$-orbit
	of the composition $\phi \colon P_g\to P_g/I\xrightarrow{\psi} \F_{q^n}$, where
	$\psi$ is an $\F_q$-algebra isomorphism $P_g/I\to \F_{q^n}$. (Here $\phi$ depends on the choice of the isomorphism
	$\psi$, but the $\Gal(\F_{q^n}/\F_q)$-orbit of $\phi$ does not).
	One easily checks that these maps are mutually inverse.
	
	A bijective correspondence between (b) and (c) is given by
	$\phi \mapsto (\phi(x_1), \dots, \phi(x_g))$. Note that by the Galois correspondence
	$(a_1,\dots,a_g)\in (\F_{q^n})^g$
	has an orbit of order $n$ if and only if $\F_q[a_1, \dots, a_g] = \F_{q^n}$.
\end{proof}

\begin{proof}[Proof of Proposition~\ref{prop.Mobius}] 
   Consider the natural (diagonal) action of $\textrm{Gal}(\F_{q^n}/\F_q)$ 
    on $(\F_{q^{n}})^g$. Let $d$ be a divisor of $n$. The group $\Gal(\F_{q^n}/\F_q)$ is cyclic
    of order $n$; its unique subgroup of index $d$ is $\Gal(\F_{q^{n}}/\F_{q^d})$.
    The  elements of $(\F_{q^n})^g$ invariant under the action of this subgroup are precisely
    the elements of $(\F_{q^{d}})^g$. Thus by Lemma \ref{lem.ideal-count}, there are $N_{g,d}(q)$ 
    $\Gal(\F_{q^n}/\F_q)$-orbits of size $d$ in $(\F_{q^n})^g$.
    Since the $\Gal(\F_{q^n}/\F_q)$-orbits partition $(\F_{q^n})^g$,
    we have
    \[q^{gn} = \sum_{d\mid n} dN_{q,d}(g) .\]
    The M\"obius inversion formula (see, e.g., \cite[Theorem 3.24]{ln}),
now yields \[ nN_{q,n}(g)  = \sum_{d\mid n} \mu(d)q^{\frac{gn}{d}} \, , \] as claimed.    
\end{proof}

\begin{corollary} \label{cor.n=1} Let $E = \F_{q} \times \dots \times \F_{q}$ ($r$ times). Then
$\gen(E) =\lceil \log_qr\rceil$.
\end{corollary}

\begin{proof}
For $n=1$, the sum in Theorem~\ref{thm.main1} reduces to just one term, $q^g$. That is,
		$\gen(E)$ is the smallest integer $g$ such that $r\leqslant q^g$. Equivalently, $\gen(E) =\lceil \log_qr\rceil$.
\end{proof}

\section{Proof of Theorem~\ref{thm.main2}}
\label{sec.explicit}

We shall need the following estimates on $N_{q, n}(g)$.

\begin{lemma}\label{lem.bounds-commutative}
    {\rm (a)}	$N_{q,n}(g)\leqslant \dfrac{1}{n}q^{gn}$ for any $n, g \geqslant 1$.


\smallskip
    {\rm (b)}	$N_{q, n}(g) \geqslant \dfrac{1}{n}q^{gn}(1-\dfrac{1}{q})$ for any $n, g \geqslant 1$.
\end{lemma}

\begin{proof}
    (a)	is an immediate consequence of Lemma~\ref{lem.ideal-count}, since the number of orbits of order $n$ in
    $(\F_{q^n})^g$ cannot exceed $\dfrac{1}{n}|(\F_{q^n})|^g$.
%
%
To prove (b) let us consider three cases. 
	
\smallskip	
	
{\it  Case	1.}  $n=1$. By Proposition~\ref{prop.Mobius}, $N_{q,n}(g)=\dfrac{1}{n} q^{g}$, and part (b) follows.

\smallskip
	
{\it  Case 2.} $n= p^e$ is a prime power, where $e \geqslant 1$. By Proposition~\ref{prop.Mobius}, 
	\begin{equation} \label{e.prime-power}
	N_{q,n}(g) ={\textstyle\dfrac{1}{n}(q^{g p^e}- q^{g p^{e-1}})= {\textstyle\dfrac{1}{n}}q^{gn}(1-q^{g(p^{e-1} -p^e)})} \, . 
	\end{equation}
	Since $p^{e-1}- p^e \leqslant -1$, part (b) follows.
   
\smallskip
 
{ \it Case 3.}	 The prime decomposition of $n$ is $n = p_1^{e_1} \dots p_m^{e_m}$,
	where $m \geqslant 2$.	In particular, $n \geqslant 6$. Let $\tau(n) = (e_1 + 1) \dots (e_m + 1)$  
	be the number of positive divisors of $n$. As $d$ ranges over these divisors, 
	the function $\mu(d)$ attains each of the values  $1$  and $-1$ exactly $2^{m-1}$ times, and in
	all other cases $\mu(d) = 0$.
	We conclude that $\mu(d) = -1$ for at most $\tau(n)/2$ divisors $d$. In other words, in the expression
	for $N_{q, n}(g)$ given by Proposition~\ref{prop.Mobius}, at most $\tau(n)/2$ terms $q^{gn/d}$ come with a negative sign.
	Since the absolute value of each of these terms is at most $q^{gn/2}$ and since $\tau(n) \leqslant 2\sqrt{n}$, 
	we see that
	\begin{equation} \label{EQ:lower-bound-M-g-n}
	N_{q,n}(g) \geqslant 
	\dfrac{1}{n}\left(q^{gn}-\sqrt{n} \, q^{\frac{gn}{2}}\right)=
	\dfrac{1}{n}q^{gn}\left(1-\sqrt{n} \, q^{-\frac{gn}{2}}\right)  \ .
	\end{equation}
	It is therefore enough to show that $\sqrt{n} \, q^{-gn/2}\leqslant \dfrac{1}{q}$,
		or equivalently, that $\sqrt{n} \, q^{1-gn/2}\leqslant 1$.
		Since $n \geqslant 6$, we have $1-\dfrac{gn}{2}<0$.	Thus,
        if the inequality $\sqrt{n} \, q^{1-gn/2}\leqslant 1$ 
        holds with $g=1$ and $q=2$, then it will hold for all $g$ and $q$. 
        Substituting $g=1$ and $q=2$, we obtain $\sqrt{n} \, 2^{1-n/2}\leqslant 1$ or equivalently,
        $2^n \geqslant 4n$. An easy induction argument shows that this inequality is satisfied
        for every $n\geqslant 4$.
\end{proof}

	\begin{proof}[Proof of Theorem~\ref{thm.main2}]
        Set $E = (\F_{q^n})^r$ and $g=\textrm{gen}(E)$. We need to  show that
		\begin{equation} \label{e.1.3a}
		{\textstyle\dfrac{1}{n}} \log_q(nr)\leqslant  g < {\textstyle\dfrac{1}{n}} \log_q(nr)+2\ .
		\end{equation}
		If $g=0$, then necessarily $r=n=1$ and the theorem holds.
		If $g=1$, then Lemmas~\ref{lem.ideals}  and~\ref{lem.bounds-commutative}(a)
		imply   $r\leqslant N_{q,n}(1) \leqslant \dfrac{1}{n}q^n$. This yields  
		$\lceil \textstyle\dfrac{1}{n} \log_q(nr)\rceil\leqslant 1$, and once again, the inequalities~\eqref{e.1.3a}
		hold. 
		
	    Thus we may assume that $g\geqslant 2$. By Lemma~\ref{lem.ideals}, $g$ is the unique integer for which
		\begin{equation*} 
            N_{q,n}(g-1) <r \leqslant N_{q,n}(g) \ .
		\end{equation*}
		By Lemma~\ref{lem.bounds-commutative}(b), this implies
		$
		\dfrac{1}{n}q^{(g-1)n}(1-\dfrac{1}{q})<r\leqslant \dfrac{1}{n}q^{gn}
		$ 
		which, after rearranging, yields
		\[{\textstyle\dfrac{1}{n}}\log_q(rn)\leqslant g 
		<{\textstyle\dfrac{1}{n}}\log_q(rn)+1+{\textstyle\dfrac{1}{n}}(1-\log_q(q-1)).
		\]
		Since $q\geqslant 2$ and $n\geqslant 1$, the right hand side cannot exceed 
		$\dfrac{1}{n} \log_q(nr)+2$.
\end{proof}

	The following corollary was stated in \cite[Remark~4.3]{fr} without proof.
	
	\begin{corollary} \label{cor.fr}
		Let $E$ be an \'etale $\F_q$-algebra. If $d=\dim_{\F_q}(E)$,
		then $\gen (E) \leqslant \lceil\log_q (d) \rceil$. 
	\end{corollary}
	
	The bound in the corollary is tight, since $\gen(E) = \lceil \log_q (d)  \rceil$ when 
	$E=(\F_q)^d$ by Corollary~\ref{cor.n=1}.

	\begin{proof}
		By Proposition~\ref{prop.pure}, we may assume without loss of generality that
		$E=(\F_{q^n})^r$ so that $d=\dim E=nr$. As we mentioned above,
	    for $n=1$, $\gen(E) = \lceil \log_q (d)  \rceil$ by Corollary~\ref{cor.n=1}. We will thus assume that $n>1$ from now on.
		By Theorem~\ref{thm.main2},		
		$\gen (E)\leqslant \lceil \dfrac{1}{n}\log_q(rn)\rceil +1 \leqslant \lceil \dfrac{1}{2}\log_q (d ) \rceil+1$.
		Thus $\gen( E)\leqslant \lceil \log_q (d) \rceil$ when $\log_q (d) >1 $.
		
		It remains to consider the case, where $\log_q (d ) \leqslant 1$, or equivalently, $nr\leqslant q$. 
		We need to show that in this case $E$ can be generated by one element. By Lemma~\ref{lem.ideals}, it suffices to prove  
		that $r\leqslant N_{q,n}(1)$. By Lemma~\ref{lem.bounds-commutative}(b), we have 
		\[ N_{q,n}(1)\geqslant \dfrac{1}{n}q^n(1 - \dfrac{1}{q}) \geqslant \dfrac{1}{n}q(q-1) \geqslant \dfrac{q}{n} \geqslant r \, , \]
		since $n> 1$ and $q\geqslant nr$. This completes the proof of the corollary.
		 \end{proof}
	
\section{The integers $N_{q, n, m}(g)$}
\label{sec:separable}

	Let $G(q,n,m)$ denote the group of $\F_q$-algebra automorphisms
	of $B := \nMat{\F_{q^n}}{m\times m}$. Let $\pi$ be the natural homomorphism
    $\pi \colon G(q,n,m)\to \Gal(\F_{q^n}/\F_q)$ given by
	restricting an element of $G(q, n, m)$ to the center $\F_{q^n}$ of $B$.
	By the Skolem--Noether Theorem (see, e.g.,~\cite[Theorem~7.1.10]{rowenII}), $\Ker(\pi)$
	is the group of inner automorphisms of $B$. That is, $\Ker(\pi) \simeq \PGL_m(\F_{q^n}):=
	\GL_m(\F_{q^n})/\F_{q^n}^\times$. Using the resulting short exact sequence of finite groups
	\[  1  \to \PGL_m(\F_{q^n})\to G(q,n,m)\to \Gal(\F_{q^n}/\F_q) \to  1 , \]
    we see that	
	\[
	|G(q,n,m)|=n\cdot |\PGL_m(\F_{q^n})|=\dfrac{n\prod_{i=0}^{m-1}(q^{nm}-q^{ni})}{q^n-1}\ 
	\]
	is the number $C$ appearing in the statement of Theorem~\ref{thm.main3}. 

\begin{defn} \label{def.Nqnm(g)}
In the sequel, $N_{q, n,m}(g)$ will denote the number of two-sided ideals $I\subset R_g=\F_q\langle X_1,\dots,X_g \rangle$ 
	for which $R_g/I\cong \nMat{\F_{q^n}}{m\times m}$ as $\F_q$-algebras.
\end{defn}

    Let $A =\nMat{\F_{q^{n}}}{m\times m}^r$.
	It is immediate from Lemma~\ref{lem.ideals}(a) that $\gen(A)$
	is the smallest integer $g$ such that
	\[
	r\leqslant N_{q,n,m}(g)\ .
	\]
	
	The following lemma is a partial extension of Lemma~\ref{lem.ideal-count} to the setting of
    separable algebras (not necessarily commutative).
	
	\begin{lemma} \label{LM:ideal-count-non-com} 
	Let $n, m$ be positive integers, $q$ be a prime power, 
	$R_g = \F_q\langle X_1, \ldots, X_g\rangle$ be the free associative algebra on $g$ generators
	over $\F_q$, and $S_g$ be the set of $g$-tuples $(a_1,\dots,a_g)\in \nMat{\F_{q^n}}{m\times m}^g$ 
	which generate $\nMat{\F_{q^n}}{m\times m}$ as an $\F_q$-algebra.
	
	The following three sets are in (pairwise) bijective correspondence. 
	In particular, each of these sets has $N_{q,n,m}(g)$ elements.
    
	\begin{enumerate}
	\item[{\rm(a)}]
	ideals $I \subset R_g$ such that $R_g/I \cong \nMat{\F_{q^n}}{m\times m}$, 
	\item[{\rm(b)}]
	orbits of $\F_q$-algebra epimorphisms $\phi:R_g\to  \nMat{\F_{q^n}}{m\times m}$ under the
	natural action of $G(q,n,m)$,
	\item[{\rm(c)}] $G(q,n,m)$-orbits in $S_g$.
	\end{enumerate}
	Moreover, every $G(q,n,m)$-orbit in $S_g$ consists
	of   $|G(q,n,m)|$ elements. 
	\end{lemma}
	
	\begin{proof}
		The bijective correspondences between (a) and (b) and between (b) and (c) are constructed in exactly 
		the same way as in the proof of Lemma~\ref{lem.ideal-count}, with the commutative polynomial ring $P_g$ 
		replaced by the free associative algebra $R_g$.  To prove the last assertion, note that if
		$a_1, \dots, a_g$ generate $\nMat{\F_{q^n}}{m\times m}$ as an $\F_q$-algebra, then the stabilizer
		of $(a_1, \dots, a_g)$ in $G(q,m,n)$ is necessarily trivial.
	\end{proof}
	
	\begin{remark} \label{rem.assoc}
		In the commutative setting of Lemma~\ref{lem.ideal-count} (i.e., for 
		$m=1$), the set $S_g$ consists precisely of the $g$-tuples $(a_1, \dots, a_g)$
		whose $G(q,n,m)$-orbit has exactly $|G(q,n,m)|$ elements. For $m \geqslant 2$, this 
		is not so in general. In other words, a $g$-tuple $(a_1,\dots,a_g)\in \nMat{\F_{q^n}}{m\times m}^g$
		with trivial stabilizer in $G(q,n,m)$ may not   generate $\nMat{\F_{q^n}}{m\times m}$.
		For example,  when $m=2$, $n=1$, and $g = 3$,  
		\[ (a_1, a_2, a_3) = \left(\begin{pmatrix} 1 & 0 \\ 0 & 0\end{pmatrix},
		\begin{pmatrix} 0 & 1 \\ 0 & 0\end{pmatrix},
		\begin{pmatrix} 0 & 0 \\ 0 & 1\end{pmatrix}\right) \]
		has trivial stabilizer in $G(q,1,2)=\PGL_2(\F_q)$. On the other hand, since $a_1$, $a_2$ and $a_3$
		are upper-triangular matrices, they do not generate $\nMat{\F_q}{2\times 2}$.
		For this reason, Lemma~\ref{LM:ideal-count-non-com} does not allow us to obtain a formula 
		for $N_{q,n,m}(g)$
	   when $m \geqslant 2$, analogous to the formula in Proposition~\ref{prop.Mobius}. However, it does lead to useful
	   estimates on $N_{q,n,m}(g)$.
	\end{remark}
	
	\begin{corollary} \label{cor.ideal-count}
	\begin{enumerate}
	 \item[{\rm(a)}] $N_{q, n, 1}(g) = N_{q, n}(g)$. 
	
	\item[{\rm(b)}] $N_{q, n, m}(g) = C^{-1}|S_g| \leqslant C^{-1} q^{gnm^2}$.
	
	\item[{\rm(c)}] Suppose $B$ is a proper $\F_q$-subalgebra of $\nMat{\F_{q^n}}{m\times m}$. Then
	$N_{q, n, m}(g) = C^{-1} |S_g| \leqslant C^{-1}(q^{gnm^2} - |B|^g)$.
	\end{enumerate}
	\end{corollary}
    
    \begin{proof} (a) The set of $N_{q, n}(g)$ elements described in Lemma~\ref{lem.ideal-count}(c) is the same as 
    the set of $N_{q, n, 1}(g)$ elements described in Lemma~\ref{LM:ideal-count-non-com}(c).
    
     (b) Follows from $S_g \subset \nMat{\F_{q^n}}{m\times m}$ and $|\nMat{\F_{q^n}}{m\times m}| = q^{gnm^2}$.

     (c) Clearly, if $a_1, \ldots, a_g \in B^g$, then $(a_1, \dots, a_g) \not \in S_g$.
     From this we see that $S_g \subset {\nMat{\F_{q^n}}{m\times m}^g \setminus B^g}$, and thus
     $|S_g| \leqslant (q^{gnm^2} - |B|^g)$.
	\end{proof}

\section{Proof of Theorem~\ref{thm.main3}}

	\begin{lemma}\label{lem.bounds-noncommutative}
%
%
$ N_{q,n,m}(g)	\geqslant  |G(q,n,m)|^{-1}{q^{(g-1)nm^2}}$ for every $g, m \geqslant 2$ and $n \geqslant 1$.
	\end{lemma}

	\begin{proof} 
		By Corollary~\ref{cor.ideal-count}(b), we need to establish the inequality $|S_g|\geqslant 2q^{(g-1)nm^2}$.
		Since $S_2\times \nMat{\F_{q^n}}{m\times m}^{g-2}\subseteq S_g$,
		it suffices to show that
	\begin{equation} \label{e.S2}
	|S_2|\geqslant  q^{nm^2}.
	\end{equation}
    To prove~\eqref{e.S2}, fix a nonzero generator $u$ of $\F_{q^n}$ over $\F_q$ and
	    consider pairs of matrices 
        \begin{equation} \label{e.matrix-pairs}	
	    A = \alpha_1 E_{1,2} + \alpha_3 E_{2,3} + \dots + \alpha_{m-1} E_{m-1, m}
	    \qquad\text{and}\qquad
	    B = \sum_{i, j = 1, \ldots, m} \beta_{ij} E_{i, j}
	    \end{equation}
	    where
	    $\alpha_i$ and $\beta_{ij}$ are arbitrary elements of $\F_{q^n}$, subject to
	    \begin{equation} 
	    \label{e.bm1} 
	    \alpha_1 \dots \alpha_{m-1} \beta_{m1} = u \, . 
	    \end{equation}
	Here, as usual, $E_{i, j}$ denotes the $(i, j)$-elementary matrix, i.e., an $m \times m$-matrix with
	$1$ in the $(i, j)$-position and zeroes elsewhere.
	 Note that $\alpha_1, \dots, \alpha_{m-1}$ can be arbitrary non-zero elements of $\F_{q^n}$.
	 Once they are chosen, $\beta_{m 1}$ is uniquely determined by~\eqref{e.bm1}. Thus the number of
	 pairs $(A, B)$ of the above form is $(q^n - 1)^{m-1}(q^n)^{m^2-1}$.

\smallskip
		{\noindent \it Claim.} Any pair of matrices $(A, B)$ defined by~\eqref{e.matrix-pairs} and~\eqref{e.bm1}
		generates $\nMat{\F_{q^n}}{m\times m}$ as an $\F_q$-algebra.

\smallskip
	
	The claim implies that the pair  $(A+\gamma I_m,B)$ generates $\nMat{\F_{q^n}}{m\times m}$
	for every $\gamma\in \F_{q }$. 
	This gives us $(q^n - 1)^{m-1}(q^n)^{m^2-1}q=(q^n-1)^{m-1}q^{n(m^2-1)+1}$ pairs of generators.
	Thus, once the claim is established,
	we can conclude that
	\[ |S_2| \geqslant (q^n - 1)^{m-1}q^{n(m^2-1)+1}
	\geq   (q^n - 1) q^{n(m^2-1)+1 }
	=(q-q^{1-n})q^{nm^2}\geq q^{nm^2} \, . \]
    This would complete the proof of~\eqref{e.S2} and thus of part (b). 
    
\smallskip
    
    We now turn to the proof of the claim. Denote by $\Lambda$ the $\F_q$-subalgebra of $\nMat{\F_{q^n}}{m\times m}$
    generated by $A$ and $B$. Our goal is to show that $\Lambda = \nMat{\F_{q^n}}{m\times m}$. We will do this in several steps.
    
\smallskip
    {\noindent \it Step 1.} For every $c \in \F_{q^n}$, $\Lambda$ contains a matrix of the form $c E_{1, 1} + c_2 E_{1, 2} + \dots + c_m E_{1, m}$
    for some $c_{2},\dots,c_m \in \F_{q^n}$.

\smallskip    
    {\noindent \it Proof.} Since $A^{m-1} = \alpha_1 \dots \alpha_{m-1} E_{1, m}$, 
    the matrix $A^{m-1} B$ has the form $u E_{1, 1} + t_2 E_{1, 2} + \dots + t_m E_{1, m}$
    for some $t_2, \dots, t_m\in \F_{q^n}$. In particular, $A^{m-1}B$ is an upper-triangular matrix with diagonal entries $u, 0, \dots, 0$.
    If $p(x)$ is a polynomial with coefficients in $\F_q$, then
    $p(A^{m-1} B) = p(u) E_{1, 1} + w_2 E_{1, 2} + \dots + w_m E_{1, m} \in \Lambda$ for some $w_2, \dots, w_m \in \F_{q^n}$.
    The desired conclusion now follows from the fact that $u$ is a generator for $\F_{q^n}$ over $\F_q$.
    
    \smallskip
    {\noindent \it Step 2.} For every $c \in \F_{q^n}$ and every $j = 1, 2, \dots, m$, $\Lambda$ contains an element of the form
    $c E_{1, j} + s_{j+1} E_{1, j+1} + \dots + s_m E_{1, m}$ for some $s_{j+1}, \dots, s_m \in \F_{q^n}$.

\smallskip    
    {\noindent \it Proof.} We argue by induction on $j$. The base case, where $j = 1$, is given by Step 1. For the induction step, assume that
    $j \geqslant 2$ and for every $c' \in \F_{q^n}$, there exist $s_j, \dots, s_m \in \F_{q^n}$ such that
    \[ L:= c' E_{1, j-1} + s_{j} E_{1, j} + \dots + s_m E_{1, m} \in \Lambda \, . \]
    Now observe that $L A \in \Lambda$ is of the form $c' \alpha_{j-1} E_{1, j} + t_{j+1} E_{1, j+1} + \dots + t_m E_{1, m}$ for some
    $t_{j+1}, \ldots, t_m \in \F_{q^n}$. Since $\alpha_{j-1} \neq 0$ by~\eqref{e.bm1}, we see that the coefficient $c' \alpha_{j-1}$
    of $E_{1, j}$ can assume an arbitrary value in $\F_{q^n}$.
    
\smallskip
    
    {\it \noindent Step 3.} $\Lambda$ contains $\F_{q^n} E_{1, j}$ for every $j = 1, \dots, m$. 

\smallskip    
    {\it \noindent Proof.} The elements given in Step 2 span $\F_{q^n} E_{1, 1} \oplus \F_{q^n} E_{1, 2} \oplus \dots \oplus \F_{q^n} E_{1, m}$ as an
    $\F_q$-vector space.

\smallskip    
    {\noindent \it Step 4.} $\Lambda$ contains $\F_{q^n} E_{k, j}$ for every $k = 1, \dots, m-1$ and $j = 1, \dots, m$. 

\smallskip    
    {\noindent \it Proof.} We argue by induction on $k$. The base case, $k = 1$, is given by Step 3. If $1 < k < m$, assume that 
    $\F_{q^n} E_{k', j} \subset \Lambda$ for every $k' = 1, \dots, k-1$ and every $j = 1, \dots, m$. 
    Subtracting a linear
    combination of $E_{1, 2}, \dots, E_{k - 1, k}$ from $A$, we see that 
    \[ A_k := \alpha_{k} E_{k, k+1} + \dots + \alpha_{m-1} E_{m-1, m} \in \Lambda \]
    and hence, so is $A_k^{m-k} = \alpha E_{k, m}$, where $\alpha = \alpha_k \cdots \alpha_{m-1}$.
    By Step 3, $t E_{1, j} \in \Lambda$ for every $t \in \F_{q^n}$ and every $j = 1, \dots, m$. 
    Consequently,
    so is $(\alpha E_{k, m}) B (t E_{1, j}) = (t \alpha \beta_{m  1}) E_{k, j}$. Since $\alpha\beta_{m  1} \neq 0$, this shows that $\F_{q^n} E_{k, j} \subset \Lambda$, 
    as required.

\smallskip    
    
   	{\it \noindent Step 5.} $\Lambda$ contains $\F_{q^n} E_{m, j}$ for every $j = 1, \dots, m$.

\smallskip    
    {\it \noindent Proof.} By Step 4, $\Lambda$ contains $E_{k, k}$ for $k = 1, \dots, m-1$. Since it also contains the identity
    element $I_m = E_{1, 1} + \dots + E_{m, m}$, we see that $E_{m, m} = I_m - E_{1, 1} - \dots - E_{m-1, m-1} \in \Lambda$.
    To show that $\F_{q^n} E_{m, j} \subset \Lambda$ for every $j = 1, \dots, m$, we will use the same 
    method as in Step 4. 
    By Step 3, $t E_{1, j} \in \Lambda$ for every $t \in \F_{q^n}$. Thus,
    \[ E_{m, m} B (t E_{1, j}) = t \beta_{m  1} E_{m, j} \in \Lambda  \, . \] 
    Since $\beta_{m  1} \neq 0$, this shows that $\F_{q^n} E_{m, j} \subset \Lambda$, as claimed. 
    
    \smallskip
Taken together, Steps 4 and 5 show that $\Lambda$ contains $\F_{q^n} E_{k, j}$ for every $k, j = 1, \dots, m$.
As a result, $\Lambda = \nMat{\F_{q^n}}{m\times m}$, which completes the proof of the claim,
 and thus of the lemma.
	\end{proof}
	
\begin{proof}[Proof of Theorem~\ref{thm.main3}]

     Let $A = \nMat{\F_{q^n}}{m\times m} \times \dots \times \nMat{\F_{q^n}}{m\times m}$ ($r$ times) and $C=|G(q,n,m)|$,
     as in the statement of the theorem. When $m=1$, Theorem~\ref{thm.main3} reduces to Theorem~\ref{thm.main2}. Thus, we may assume that $m>1$.

    By Lemma~\ref{lem.ideals}, $\gen(A)$ is the unique integer $g$ satisfying
	\begin{equation} \label{e.ideal-count} N_{q,n,m}(g -1) < r\leqslant N_{q,n,m}(g) \ . \end{equation}
     Our goal is to show that 
     \begin{equation} \label{e.main3}
	{\textstyle\dfrac{1}{nm^2}}\log_q(C \cdot r) 
	\leqslant 
	g <  {\textstyle\dfrac{1}{nm^2}}
	\log_q(C \cdot r)+2\ .	
	\end{equation}
	
	Suppose $g \geqslant 3$. By Lemma~\ref{lem.bounds-noncommutative} and Corollary~\ref{cor.ideal-count}(b),
	\eqref{e.ideal-count} implies
	\[
	C^{-1}q^{(g-2)nm^2}  <  r \leqslant q^{gnm^2} C^{-1} \ .
	\] 	
	Multiplying through by $C$ and taking $\log_q$, we obtain~\eqref{e.main3}.
	
	Since $A$ is non-commutative, $g\leqslant 1$ is impossible.
	Thus, it remains to consider the case $g=2$. In this case, Corollary~\ref{cor.ideal-count}(b)
	yields $r \leqslant N_{q, n, m}(g) \leqslant C^{-1} q^{g n m^2}$, and thus
	${\textstyle\dfrac{1}{nm^2}}\log_q(r\cdot C) 
	\leqslant 
	g$.	On the other hand, the upper bound on $g$ from~\eqref{e.main3} also remains
	valid, since $g =2 < {\textstyle\dfrac{1}{nm^2}}\log_q(C \cdot r)+2$.
\end{proof}

\section{Proof of Theorem~\ref{thm.main4}}	
\label{sec:I}

We begin with further estimates on $N_{q, n, m}(g)$. We will denote $|G(q, n, m)|$ by $C$
throughout, as in the statement of Theorem~\ref{thm.main3}.

\begin{lemma} \label{lem.vector-space}
Let $V$ be a $d$-dimensional vector space over $\F_q$ and let $T_g \subset V^g$ be the set of $g$-tuples
$(v_1, \ldots , v_g)$ which   span $V$. Then 
$|T_g|=q^{gd}(1-O(q^{-g}))$.
\end{lemma}

\begin{proof}
	The cardinality of $T_g$ is the equal to the number of   matrices  in $\nMat{\F_q}{d\times g}$
     of rank $d$. When  $g\geqslant d$ this is the set of $d\times g$ matrices over $\F_q$
     whose rows are linearly independent. The
     cardinality of this set is well-known to be 
     $\prod_{i=0}^{d-1}(q^g-q^i)=q^{dg}\prod_{i=0}^{d-1}(1-q^{i-g})\geqslant q^{dg}(1-\sum_{i=0}^{d-1}q^{i-g})$.
     Since $|T_g|\leqslant q^{dg}$, the lemma follows.
\end{proof}

\begin{lemma} \label{lem.N-g-m-bounds}

\begin{enumerate}
%
\item[{\rm(a)}]
$N_{q, n, m}(g) = C^{-1}   q^{gnm^2} (1-O(q^{-g}))$.

\item[\rm{(b)}] If $n \geqslant 2$, then
$N_{q, n, m}(g) \leqslant C^{-1} (q^{gnm^2} - q^{gm^2})$. 

\item[{\rm(c)}] If $m \geqslant 2$, then
$N_{q, n, m}(g) \leqslant C^{-1} (q^{gnm^2} - q^{gn(m^2 - m + 1)})$. 
\end{enumerate}
\end{lemma}

     \begin{proof} (a) As before, let $S_g$ be the set of $g$-tuples 
     in $\nMat{\F_{q^n}}{m\times m}^g$ which generate $\nMat{\F_{q^n}}{m\times m}$ as an
     $\F_q$-algebra. By Corollary~\ref{cor.ideal-count}(b), $N_{q, n, m}(g) = C^{-1}|S_g|$.      
     We now apply
     Lemma~\ref{lem.vector-space} with $V = \nMat{\F_{q^n}}{m\times m}$ and $d = nm^2$. Clearly every
     $g$-tuple that spans $\nMat{\F_{q^n}}{m\times m}$ as an $\F_q$-vector space
     also generates it as an $\F_q$-algebra. Thus, $|S_g| \geqslant  |T_g|=q^{gnm^2}(1-O(q^{-g}))$,
     and so $N_{q, n, m}(g) \geqslant C^{-1}   q^{gnm^2} (1-O(q^{-g}))$.
     On the other hand, $N_{q,n,m}(g)\leqslant C^{-1}   q^{gnm^2}$
     by  Corollary~\ref{cor.ideal-count}(b). 
     
     (b) follows from Corollary~\ref{cor.ideal-count}(c) with $B =\nMat{\F_{q}}{m\times m}$.
     
     (c) Let $B$ be the subalgebra of $\nMat{\F_{q^{n}}}{m\times m}$ consisting of matrices with zeroes in 
     positions $(2,1), \dots, (m,1)$, and apply Corollary~\ref{cor.ideal-count}(c). Note 
     that $|B| = q^{n(m^2-m + 1)}$.
\end{proof}

	\begin{proof}[Proof of Theorem~\ref{thm.main4}]     
     We first claim that
    \begin{equation} \label{e.main4i}
        I_0(g)=\N\cap (C^{-1}q^{(g-1)nm^2},N_{q,n,m}(g)]
    \end{equation}
    and
    \begin{equation} \label{e.main4ii}
    I_1(g) = \N \cap (N_{q,n,m}(g-1),C^{-1}q^{(g-1)nm^2}].
    \end{equation}
    Indeed, by Lemma~\ref{lem.ideals}, $\gen(A_r) = g$
     if and only $N_{q, n, m}(g-1) < r \leqslant N_{q,n,m}(g)$. On the other hand, 
     by Theorem~\ref{thm.main3}, $\gen(A_r) = g$ if and only if $r \in I_0(g) \sqcup I_1(g)$.
     Thus,
     \[ I_0(g) \sqcup I_1(g) = (N_{q, n, m}(g-1), N_{q, n, m}(g)] \, , \]
     and \eqref{e.main4ii} follows from \eqref{e.main4i}. To establish 
     \eqref{e.main4i}, note that $r > C^{-1} q^{(g-1)n m^2}$
     is equivalent to $g < \dfrac{1}{nm^2} \log_q(C \cdot r) + 1$. This proves the claim.
    
    We now turn to the proof of itself. 
    Equations \eqref{e.main4i} and \eqref{e.main4ii} tell us that
    $I_0(g) = \N \cap (C^{-1} q^{(g-1)nm^2}, N_{q, n, m}(g)]$
    and $I_1(g+1) = \N \cap (N_{q, n, m}(g), C^{-1} q^{n m^2}]$. These  intervals are, by definition,
    disjoint, and part (a) follows. 
    
    To prove part (b), we combine~\eqref{e.main4i} with Lemma~\ref{lem.N-g-m-bounds}(a):
     \begin{align*}
     |I_0(g)|&=N_{q,n,m}(g)-C^{-1}q^{(g-1)nm^2}-1+O(1)\\
     &=
     C^{-1}q^{gnm^2}(1-O(q^{-g}))-C^{-1}q^{(g-1)nm^2}-1+O(1)\\
     &=
     C^{-1}(q^{gnm^2}-q^{(g-1)nm^2})(1-O(q^{-g})) \ .
     \end{align*}
     To prove (c), we combine~\eqref{e.main4ii} 
     with Lemma~\ref{lem.N-g-m-bounds}(b):
      \begin{align*}
     |I_1(g)|&\geqslant 
     \lfloor C^{-1}q^{(g-1)nm^2}-N_{q,n,m}(g-1)\rfloor \geqslant \lfloor C^{-1}q^{gm^2} \rfloor \ .
     \end{align*}
     Similarly, to prove (d), we combine~\eqref{e.main4ii} with Lemma~\ref{lem.N-g-m-bounds}(c):
     \[
     |I_1(g)|\geqslant
     \lfloor C^{-1}q^{(g-1)nm^2}-N_{q,n,m}(g-1) \rfloor \geqslant \lfloor C^{-1}q^{gn(m^2-m+1)} \rfloor \ . 
		\qedhere    
     \]
     \end{proof}

\end{document}